\documentclass[11pt,reqno]{amsart}
\usepackage{xifthen,setspace,bm}
\usepackage[utf8]{inputenc} 
\usepackage[T1]{fontenc} 
\usepackage{pkgfile}

\allowdisplaybreaks

\title{Upper Tails for edge eigenvalues of Random Graphs}

\author[Bhattacharya]{Bhaswar B. Bhattacharya}
\address{B.\ B.\ Bhattacharya\hfill\break
	Department of Statistics\\ University of Pennsylvania\\ Philadelphia, PA 19104, USA.}
\email{bhaswar@wharton.upenn.edu}

\author[Ganguly]{Shirshendu Ganguly}
\address{S. Ganguly \hfill\break
	Department of Statistics\\ UC Berkeley \\ 
	Berkeley, California, CA 94720, USA.}
\email{sganguly@berkeley.edu}

\begin{document}

\maketitle

\begin{abstract}  The upper tail problem for the largest eigenvalue of the Erd\H{o}s--R\'enyi random graph $\mathcal{G}_{n,p}$ is to estimate the probability that the largest eigenvalue of the adjacency matrix of $\mathcal{G}_{n,p}$ exceeds its typical value by a  factor of $1+\delta$. In this note we show that  for $\delta >0$ fixed, and $p \rightarrow 0$ such that $n^{\frac{1}{2}} p \rightarrow \infty$, the upper tail probability for the largest eigenvalue of $\mathcal{G}_{n,p}$ is 
$$\exp\left[-(1+o(1)) \min\left\{\tfrac{(1+\delta)^2}{2}, \delta(1+\delta) \right\} n^{2}p^{2}\log (1/p)\right].$$ 
In the same regime of $p$, we show that the second largest eigenvalue $\lambda_2(\cG_{n,p})$ of  the adjacency matrix of $\mathcal{G}_{n,p}$ satisfies  
$$\P(\lambda_2(\cG_{n,p})\ge  \delta np) = \exp\left[-(1+o(1)) \tfrac{1}{2} \delta^2n^2p^2 \log (1/p) \right],$$
where $\delta =\delta_n < 1$ can depend on $n$ such that $\delta n^{\frac{1}{2}} p \rightarrow \infty$, which covers deviations of $\lambda_2(\cG_{n,p})$ between $n^{\frac{1}{2}}$ and $np$.  
Our arguments build on recent results on the large deviations of the largest eigenvalue and related non-linear functions of the adjacency matrix  in terms of natural mean-field entropic variational problems.
\end{abstract}

\maketitle

\section{Introduction}

Establishing large deviations for edge eigenvalues (and other spectral functionals) of random matrices is an important problem, which in most cases is widely open. Even in the basic example of self-adjoint Wigner matrices with i.i.d. entries, the question is well-understood only in a few special cases. These include the integrable Gaussian models namely the GOE or GUE where the entries are i.i.d. real and complex Gaussians respectively, crucially relying on the exact joint density of eigenvalues available for these models enabling the usage of Coulomb gas methods. This can be used to exactly compute  the rate function for the large deviation principle (LDP) for the empirical spectral measure \cite{freeentropy} as well as the spectral norm \cite{aging}.  For results in the related setting of Gaussian covariance or Wishart matrices see \cite{satya1,satya2}. 
However, for the non-Gaussian cases, the results have been few and far between. Bordenave and Caputo in \cite{bordenave1} studied LDP for empirical spectral measure for Wigner matrices with stretch exponential tails while the corresponding results for the spectral norm were obtained in \cite{augeri2}. 
Very recently, Guionnet and Husson \cite{guionnet1} established an LDP for the
largest eigenvalue of Rademacher matrices and more generally sub-Gaussian Wigner matrices.

In this note, we consider the upper tail large deviation problem for the extremal/edge eigenvalues of adjacency matrices of sparse  random graphs. To this end, let $\cG_{n,p}$ be the Erd\H os-R\'enyi random graph on $n$ vertices with edge probability $p$. Denote by $\lambda_1(\cG_{n,p}) \geq \lambda_2(\cG_{n,p}) \geq \cdots \geq \lambda_n(\cG_{n,p})$ the eigenvalues of the adjacency matrix $A(\cG_{n,p})$ of $\cG_{n,p}$, arranged in non-increasing order. Here, we particularly focus on the largest and the second largest eigenvalues of $\cG_{n,p}$, namely $\lambda_1(\cG_{n,p})$ and $\lambda_2(\cG_{n,p})$. The typical behaviors of  the edge 
eigenvalues of $\cG_{n,p}$ are known to great precision, for instance, with high probability $ \lambda_1(\cG_{n,p})=(1+o(1)) np$, and $ \lambda_2(\cG_{n,p})=(1+o(1)) \sqrt{np}$, for $p$ not too sparse (see, for example, \cite{yau} and the references therein). General concentration inequalities for the edge eigenvalues of $\cG_{n, p}$ are also well-known \cite{AKV} (see \cite{ev_concentration} for a recent improvement for the largest eigenvalue).

The study of large deviations of random graphs, in spite of being a fairly recent topic, have witnessed a series of breakthroughs over the last decade starting with the seminal work of Chatterjee and Varadhan \cite{CV11} who considered subgraph counts in $\cG_{n,p}$ in the dense regime, that is, $p \in (0, 1)$ is fixed. As a consequence in \cite{CV12}, they proved a large deviations principle for the entire spectrum of $\cG_{n, p}$ thought of as a  countable ordered sequence, with the topology of coordinate wise convergence, at scale  $np$. The problem turned out to be much more challenging in the sparse regime, where $p=p(n) \rightarrow 0$. This was first addressed in a systematic way in the tour-de-force work of Chatterjee and Dembo \cite{CD16}.  Here, the authors explored the more general problem of  approximating the partition function for a general Gibbs measure on the hypercube with the so-called `mean-field' variational  principle; and came up with a notion of complexity of the gradient of the Hamiltonian (of the Gibbs measure) along with additional smoothness properties under which the above approximation holds.  This as a direct consequence reduced the large deviations for subgraph counts in $\cG_{n,p}$ to a `mean-field' entropic variational problem in certain regimes of the sparsity parameter $p$. Thereafter, Eldan \cite{Eldan} obtained an improved set of conditions under which the above reduction holds, approximating the Gibbs measure, up to lower order entropy, by a product measure. Similar results for Gibbs measures beyond the hypercube were obtained by \cite{BM17,yan}, and recently by Austin \cite{austin} for very general product spaces. Further improvements on the range of the sparsity parameter for several spectral and geometric functionals were established independently and simultaneously by Cook and Dembo \cite{CD18} and Augeri \cite{augeri1}. In particular, Augeri \cite{augeri1} establishes the validity of the mean-field variational problem for the upper tail of triangles in $\cG_{n,p}$, for $p \gg n^{-\frac{1}{2}}$.\footnote{We write
 $f \lesssim g$ to denote $f = O(g)$, $f \sim g$ means $f = (1 + o(1))g$, and $f \ll g$ means
$f = o(g)$.} The validity of the mean-field variational problem down to the sparsity of $n^{-1}$ (up to logarithmic factors) has been proved for triangles and counts of non-bipartite graphs by Harel et al. \cite{upper_tail_localization} and, very recently, for counts of general cycles  by Basak and Basu \cite{ab_rb}.\footnote{For $p \lesssim n^{-1}$ up to poly-logarithmic terms, the problem goes through a qualitative change with observables such as cycle counts falling in the Poisson-regime, which dictates the large deviation behavior, and the mean-field approach no longer offers the right perspective. Even though beyond the scope of this paper, obtaining precise bounds on the upper-tail probability in this regime continues to be an important problem of study (see, for example, \cite[Theorem 1.6]{upper_tail_localization} and \cite{warnke_missing_log} for some recent progress and related results).}

Almost in parallel, fortunately, many of the associated entropic variational problems in the space of weighted graphs have been precisely analyzed in a separate series of papers: starting with the upper tail for triangle (more generally any regular subgraph) counts in the dense regime \cite{LZ-dense}, followed by upper tails of clique counts in the sparse case \cite{LZ-sparse}, followed by upper tails of general subgraph counts in sparse case \cite{BGLZ}. The corresponding problem for the lower tail was studied in \cite{Zhao-lower}, where several questions remain open. The analogous problem for the number the arithmetic progressions in a random set was recently studied in \cite{BGSZ}.  

In this work, building on the above recent results of both kinds, we  pin down the exact upper tail rate function for certain regimes of large deviations for  $\lambda_1(\cG_{n,p})$, the operator norm of the centered adjacency matrix $A(\cG_{n,p})- p \bm 1 \bm 1'$,\footnote{For any matrix $Q$, we denote by $Q'$ the transpose of $Q$. Then, note that $A(\cG_{n,p})-p \bm 1 \bm 1'$ is the centered adjacency matrix up to a translation by $p\bI$, where $\bI$ is the identity matrix. Since this deterministically shifts every eigenvalue by $p$ which is asymptotically negligible, we will continue this abuse of terminology.} and the second largest eigenvalue  $\lambda_2(\cG_{n,p})$, in the sparse setting. The proof relies on the connection between the number of cycles in a graph and the spectral moments of the adjacency matrix, and borrows techniques from \cite{BGLZ} about the solution of the associated variational problems for large deviations of subgraph counts. As will be evident, in a sense made precise in this note, the large deviation for the extremal eigenvalues are dictated by low rank deformations.  We leave several interesting questions open for further research.

\subsection{Statement of Results} Let $\sG_n$ denote the set of weighted undirected graphs on $n$ vertices with edge weights in $[0, 1]$, that is, if $A(G)$ is the adjacency matrix of $G$ then
\[
\sG_n =\left\{G_n:  A(G_n)=(a_{ij})_{1\leq i, j \leq n},\, 0\leq a_{ij}\leq 1,\, a_{ij} = a_{ji},\, a_{ii} = 0 \text{ for all } i, j\right\}.
\] 
For $G_n \in \sG_n$, denote by $\lambda_1(G_n) \geq \lambda_2(G_n) \geq \cdots \lambda_n(G_n)$ the eigenvalues of $G_n$ in non-increasing order. Moreover, for $G_n \in \sG_n$, $I_p(G_n)$ is the entropy relative to $p$, that is,
\[I_p(G) :=\sum_{1\leq i < j\leq n}I_p(a_{ij}) \quad \text{where} \quad I_p(x):=x\log\frac{x}{p}+(1-x)\log \frac{1-x}{1-p}\,.\]  

It is well-known that  $\E \lambda_1(\cG_{n,p})=(1+o(1)) np$, for $p \gg n^{-1}$ (up to polynomial factors in $\log n$), Recently, Cook and Dembo \cite[Proposition 1.13]{CD18} established the upper tail large deviations for $\lambda_1(\cG_{n,p})$, in terms of an entropic variational problem. 

\begin{thm}[Cook and Dembo \cite{CD18}] \label{thm:e1_ldp} For any $\theta >1$ fixed and  $n^{-\frac{1}{2}}  \ll p \leq \frac{1}{2}$,
$$-\log \P(  \lambda_1(\cG_{n,p}) \geq \theta np) = (1+o(1)) \phi_1(n ,p, (1+o(1))\theta),
$$
where 
\begin{equation}\label{eq:disvare1} 
\phi_1(n, p, \theta):= \inf\left\{I_p(G_n) : G_n\in \sG_n \text{ with }  \lambda_1(G_n) \geq \theta n p \right\}. 
\end{equation} 
\end{thm}

Here, we complement this result by solving the variational problem $\phi_1(n, p, \theta)$, in the {\it sparse regime}, that is, $p \ll 1$, which combined with the result above gives the precise upper tail asymptotics for $\lambda_1(\cG_{n,p})$:

\begin{thm}\label{thm:ev1}
For any $\delta >0$ fixed and $n^{-\frac{1}{2}} \ll p \ll 1$, 
\begin{align}\label{eq:ev1}
\lim_{n \rightarrow \infty}\frac{-\log \P(\lambda_1(\cG_{n, p}) \ge (1+\delta)np)}{n^{2}p^{2}\log (1/p)}=
\min\left\{\frac{(1+\delta)^2}{2}, \delta(1+\delta) \right\}.
\end{align}
\end{thm}

The proof of the Theorem is presented in Section \ref{sec:pfe1}. This is done in two steps: For the lower bound on the variational problem, our strategy is to bound the largest eigenvalue $\lambda_1(\cG_{n, p})$ with the Schatten $s$-norm of $\cG_{n, p}$, where $s$ is even.\footnote{For a $n\times n$ symmetric matrix $X$, the Schatten $s$-norm is defined as $||X||_s:=\left(\sum_{i=1}^s |\lambda_i|^s\right)^{\frac{1}{s}}$, where $\lambda_1, \lambda_2, \ldots, \lambda_n$ are the eigenvalues of $X$.}  Then, using the correspondence between the Schatten $s$-norm and the homomorphism density of the $s$-cycle in $G_n$, and analyzing the large $s$ behavior of the corresponding variational problem for $s$-cycles, derived in \cite{BGLZ}, the lower bound follows. The matching upper bound is attained by planting in $\cG_{n, p}$ either a {\it clique} (a set of vertices connected between themselves), or an {\it anti-clique} (a set of vertices connected to every other vertex in the graph)  of the required size. 

\begin{remark} Note that Theorem \ref{thm:ev1} above holds only in the sparse regime, $p \ll 1 $. The dense regime, that is, $p \in (0, 1)$ is fixed, falls in the purview of graphon theory \cite{Lov12}, and the large deviation principle of $\lambda_1(\cG_{n, p})$ is a consequence of the general framework of Chatterjee and Varadhan \cite{CV11,CV12}. The associated variational problem was analyzed by Lubetzky and Zhao \cite{LZ-dense}. They obtained the exact region of {\it replica symmetry}, that is, the set of values of $(p, r)$, where $r=(1+\delta)p$, for which the constant function uniquely minimizes the associated variational problem \cite[Theorem 1.2]{LZ-dense}.
\end{remark}

Next, we study the large deviations of the second largest eigenvalue of $\cG_{n, p}$. To this end, we begin by considering the operator norm of the centered adjacency matrix, that is, $A(\cG_{n,p})-p \bm 1 \bm 1'$, a problem of independent interest.\footnote{Recall that for a $n \times n$ symmetric matrix $A$, the operator norm is defined as: $||A||_{\mathrm{op}}=\sup_{||x||=1} ||A x||$.} The following mean field approximation was proved in \cite{CD18}.
\begin{thm}[Cook and Dembo \cite{CD18}]\label{secondeigcookdembo} For $\frac{\log n}{n} \lesssim p \leq \frac{1}{2}$ and $\theta \gg n^{\frac{1}{2}}$, 
$$-\log \P( ||A(\cG_{n,p})-p \bm 1 \bm 1'||_{\mathrm{op}} \geq \theta) = (1+o(1)) \phi_2(n ,p, \theta+o(\theta)),$$
where 
\begin{equation}\label{eq:disvarlambda}
\phi_2(n, p, \theta):= \inf\left\{I_p(G_n) : G_n\in \sG_n \text{ with }  ||A(G_n)-p \bm 1 \bm 1'||_{\mathrm{op}} \geq \theta \right\}. 
\end{equation} 
\label{thm:center_ldp}
\end{thm}

\begin{remark} Note that the above result gives the upper tail asymptotics for $||A(\cG_{n,p})-p \bm 1 \bm 1'||_{\mathrm{op}}$ for deviations growing faster than $n^{\frac{1}{2}}$. {On the other hand, recall that the typical value of $||A(\cG_{n,p})-p \bm 1 \bm 1'||_{\mathrm{op}}$ is $(1+o(1)) \sqrt{np}$.} As alluded to earlier, recently, Guionnet and Husson \cite{guionnet1} established a large deviations principle for the
largest eigenvalue of $n$-dimensional Wigner matrices, rescaled by $n^{-\frac{1}{2}}$, whose independent, standardized entries have uniformly sub-Gaussian moment generating functions (which allow for Rademacher entries). Observe that even though $A(\cG_{n,p})-\E(A(\cG_{n,p}))$ is a Wigner matrix, such uniform sub-Gaussian domination does not apply in the case when $p \ll 1$. Therefore, establishing the large deviations for  $||A(\cG_{n,p})-p \bm 1 \bm 1'||_{\mathrm{op}}$ in the scale $\sqrt{np}$ in the sparse regime ($p \ll 1$), where the mean-field variational problem \eqref{eq:disvarlambda} no longer determines the rate of the upper tail,  remains open.  
\end{remark}

As in Theorem \ref{thm:ev1}, we can complement Theorem \ref{thm:center_ldp} by solving the variational problem $\phi_2(n, p, \theta)$, in the sparse regime, proving the precise upper tail asymptotics for $ ||A(\cG_{n,p})-p \bm 1 \bm 1'||_{\mathrm{op}}$, for deviations growing faster than $n^{\frac{1}{2}}$.

\begin{thm}\label{thm:ev1_centered} 
{For $n^{-\frac{1}{2}} \ll p \ll 1$ and $\delta=\delta_n=O(1)$  such that $\delta n p \gg n^{\frac{1}{2}}$},
\begin{align}\label{eq:ev1_centered_var}
\phi_2(n, p, \delta n p)= (1+o(1)) \tfrac{1}{2}\delta^2 n^{2}p^{2}\log (1/p). 
\end{align}
This implies, under the same conditions on $p$ and $\delta$ as above,
\begin{align}\label{eq:ev1_centered}
\lim_{n \rightarrow \infty}\frac{-\log \P( ||A(\cG_{n, p})-p \bm 1 \bm 1'||_{\mathrm{op}}\ge \delta np )}{\tfrac{1}{2} \delta^2n^2p^2 \log (1/p)}= 1. 
\end{align} 
\end{thm}

The proof of the theorem is presented in Section \ref{sec:pfev1_centered}. 
For the lower bound on  $\phi_2(n, p, \delta n p)$, we bound the operator norm $||A(\cG_{n, p})-p \bm 1 \bm 1'||_{\mathrm{op}}$ with the Frobenius norm, which can then be used to approximate the entropy function, when $p \ll 1$. In this case, the matching upper bound is attained by planting in $\cG_{n, p}$ a {\it clique} of the required size.

\begin{remark} Note that, unlike in Theorem \ref{thm:ev1}, in Theorem \ref{thm:ev1_centered} above and Corollary \ref{cor:ev2} below, $\delta=\delta_n$ is allowed to go to zero with $n$. In fact, the condition on $\delta$ inherited from Theorem \ref{secondeigcookdembo} implies $ n^{\frac{1}{2}} \ll \delta np \lesssim np$, covering all possible deviations of the operator norm from $\sqrt{n}$ to $O(np)$. 
\end{remark}

Finally, we obtain as a corollary the upper tail rate function for $\lambda_2(\cG_{n, p})$ in a certain regime. The probability upper bound relies on eigenvalue interlacing and the previous result while the lower bound relies on a well known variational representation of the second largest eigenvalue.

\begin{cor}\label{cor:ev2}{For $n^{-\frac{1}{2}} \ll p \ll 1$ and $\delta =\delta_n < 1$ such that $\delta n p \gg n^{\frac{1}{2}}$}, 
\begin{align}\label{eq:ev2}
\lim_{n \rightarrow \infty}\frac{-\log \P(\lambda_2(\cG_{n,p})\ge  \delta np)}{\tfrac{1}{2} \delta^2n^2p^2 \log (1/p)}= 1. 
\end{align}
\end{cor}

Observe that the above result is stated only for $\delta<1$. The following remark gives a  brief, high level discussion of the reason for such a restriction.  

\begin{remark}
Note that typically $\lambda_1(\cG_{n,p}) \approx np$ and the corresponding Frobenius eigenvector is `approximately' the constant vector $\bf{1}$, while  $\lambda_2(\cG_{n,p})\approx \sqrt{np}$.  A key step involved in our analysis of the large deviation properties of the latter, is to project the matrix $A(\cG_{n, p})$ on to the orthogonal complement of $\bf{1}$, and by Weyl's inequality reducing the large deviation problem of  $\lambda_2(\cG_{n,p})$ to a large deviation problem of the largest eigenvalue of the projected matrix, and thereafter relying on Theorem \ref{thm:ev1_centered}. However, such a reduction is tight only if $\bf{1}$ is an approximate Frobenius eigenvector even after forcing $\lambda_2(\cG_{n,p})\ge \delta np.$ Not surprisingly, it turns out that the latter constraint is approximately the same as forcing $\lambda_2\approx \delta np$, while trivially this also implies $\lambda_1(\cG_{n,p}) \ge \delta np.$ Thus, the condition $\delta< 1$ ensures that the atypical behavior for $\lambda_2(\cG_{n,p}),$ does not cause $\lambda_1(\cG_{n,p})$ and its leading eigenvector to behave atypically (they are still approximately $np$ and $\bf{1}$, respectively), making the above proof strategy yield tight results. The technical details, presented in Section \ref{sec:pfev2}, and also the discussion on open problems in Section \ref{sec:openproblems} provide further elucidation. 
\end{remark}

\section{Proof of Theorem \ref{thm:ev1}}
\label{sec:pfe1}

By Theorem \ref{thm:e1_ldp}, to show \eqref{eq:ev1}, it suffices to prove 
\begin{align}\label{eq:ev1_variational}
\lim_{n \rightarrow \infty}\frac{\phi_1(n, p, 1+\delta)}{n^{2}p^{2}\log (1/p)}=
\min\left\{\frac{(1+\delta)^2}{2}, \delta(1+\delta) \right\}.
\end{align}
To see this, note that formally, to use Theorem \ref{thm:e1_ldp}, one has to analyze $\tfrac{\phi_1(n, p, (1+o(1))1+\delta)}{n^{2}p^{2}\log (1/p)}$. However, a simple monotonicity argument shows why \eqref{eq:ev1_variational} suffices: To  this end, denote $c(\delta):=\min\{\tfrac{1}{2}(1+\delta)^2, \delta(1+\delta) \}$. Now, fix $\delta_1< \delta < \delta_2$, and take $n$ large enough such that $1+\delta_1< (1+o(1)) (1+\delta) < 1+\delta_2$.  Observe that, by definition \eqref{eq:disvare1}, $\phi_1(n, p, \theta)$ is monotone in $\theta$, which implies 
$$\phi_1(n, p, (1+\delta_1)) \leq \phi_1(n, p, (1+o(1))(1+\delta)) \leq \phi_1(n, p, (1+\delta_2)).$$ Therefore, taking limit as $n \rightarrow \infty$ and \eqref{eq:ev1_variational} gives, 
$$c(\delta_1) \leq \liminf_{n \rightarrow \infty}\frac{ \phi_1(n, p, (1+o(1))(1+\delta))}{n^{2}p^{2}\log (1/p)}\le \limsup_{n \rightarrow \infty}\frac{\phi_1(n, p, (1+o(1))(1+\delta))}{n^{2}p^{2}\log (1/p)}\leq  c(\delta_2).$$
Since $\delta_1$ and $\delta_2$ can be chosen to be arbitrarily close of $\delta$, and the function $c(\cdot)$ is continuous, we get
$$\lim_{n \rightarrow \infty}\frac{\phi_1(n, p, (1+o(1))(1+\delta))}{n^{2}p^{2}\log (1/p)}=
c(\delta),$$
as required.

The proof of the lower bound, which involves bounding the largest eigenvalue with the  homomorphism density of cycles, is given below in Section \ref{sec:lb_e1}. The upper bound constructions are presented in Section \ref{sec:ub_e1}. 

\subsection{Proof of the Lower Bound in \eqref{eq:ev1_variational}}
\label{sec:lb_e1}

For any fixed graph $H$ on $[k]:=\{1, 2, \ldots, k\}$ vertices and $G_n \in \cG_n$, 
\begin{equation}\label{homonotiso}
t(H, G_n):= \frac{1}{n^{|V(H)|}} \sum_{1\leq i_1,\cdots,i_k\leq n}\prod_{(x, y)\in E(H)}a_{i_x i_y}
\end{equation}
will denote the homomorphism density of $H$ in $G_n$. The negative logarithm of the upper tail large deviation probability for $t(H, \cG_{n, p})$, that is $-\log \P(t(H, \cG_{n, p}) \geq \theta n^s p^s)$
 is given by the following variational problem (see \cite[Corollary 1.6]{CD18} and \cite{BGLZ} for details):  
\begin{equation}
\phi(C_s, n, p, \theta):= \inf\left\{I_p(G_n) : G_n\in \sG_n \text{ with } t(C_s, \cG_{n, p}) \geq \theta n^s p^s\right\},
\label{eq:disvar}
\end{equation}
for $\theta >1$ and certain sparsity conditions on $p$. The solution of this variational problem in the sparse regime was established in \cite{BGLZ}, which we state below for the special case of cycles. To this end, for $s \geq 1$, let $C_s$ denote the cycle of length $s$. 

\begin{thm}\label{thm:discrete-var}\cite{BGLZ}
For any $\theta >1$ and $n^{-\frac{1}{2}} \ll p \ll 1$, the solution to the variational problem~\eqref{eq:disvar} satisfies
\[\lim_{n \to \infty} \frac{\phi(C_s,n, p , \theta)}{n^2p^2\log(1/p)}= \min\left\{\gamma\,,\tfrac12 (\theta-1)^{\frac{2}{s}}\right\}, 
\]
where $\gamma=\gamma(C_s, t)$ is the unique positive solution to $P_{C_s}(\gamma)=t$, with $P_{C_s}(\cdot)$ is the independence polynomial of $C_s$.\footnote{The \emph{independence polynomial} of a graph $H$ is defined to be $P_H(x) := \sum_k i_H(k) x^k$, where $i_H(k)$ is the number of $k$-element independent sets in $H$.}
\end{thm}

Let $G_n \in \sG_n$ be a weighted graph satisfying  $\lambda_1(G_n) \geq (1+\delta) np$. We can now use the above result to obtain a lower bound on \eqref{eq:ev1_variational}, using the following simple inequality: For  $s \geq 1$ even, 
$$t(C_s, G_n) = \frac{1}{n^s} \sum_{j=1}^n \lambda_j^s(G_n) \geq \left(\frac{\lambda_1(G_n)}{n}\right)^s \geq  (1+\delta)^s p^s.$$ 
This implies, 
\begin{align*}
\phi_1(n, p, 1+\delta) & \geq \phi(C_s, n, p, (1+\delta)^s). 
\end{align*} 
Therefore, by Theorem \ref{thm:discrete-var}, it follows that, for any $s$ even, 
\begin{align}\label{eq:lb_I}
\liminf_{n \rightarrow \infty} \frac{\phi_1(n, p, 1+\delta)}{n^2p^2\log(1/p)} \geq 
\liminf_{n \rightarrow \infty} \frac{\phi(C_s, n, p, (1+\delta)^s)}{n^2p^2\log(1/p)} & = \min\left\{\bar \gamma, \dfrac{1}{2}\left((1+\delta)^s-1\right)^{\frac{2}{s}}\right\}. 
\end{align}
where $\bar \gamma=\bar \gamma(C_s, \delta)$ is the unique positive solution to $P_{C_s}(\bar \gamma)=(1+\delta)^s$, with $P_{C_s}(\cdot)$ as defined in \eqref{eq-P_Cs}.

Therefore, to complete the proof of the lower bound in \eqref{eq:ev1_variational}, it suffices to understand the asymptotics of \eqref{eq:lb_I}, as $s \rightarrow \infty$. In turns out that the independence polynomial for cycles can be calculated in closed form in terms of Chebyshev polynomials, using the recursion\footnote{By the definition of the independence polynomial, for any graph $H$ and vertex $v$ in it, $P_{H}(x)= P_{H_1}(x) + x P_{H_2}(x)$, where $H_1$ is obtained from $H$ by deleting $v$ and $H_2$ is obtained from $H$ by deleting $v$ and all its neighbors.}
\[ P_{C_s}(x) = P_{C_{s-1}}(x) + x P_{C_{s-2}}(x)\,,\quad P_{C_2}(x) = 2x+1 \,,\quad P_{C_3}(x) = 3x+1\,.\] 
Then using the closely related recursion for Chebyshev polynomials gives, 
$$P_{C_s}(x)  =\frac{1}{2^{s-1}}\sum_{a=0}^{\lfloor \frac{s}{2}\rfloor} \binom{s}{2 a} (1+4x)^a,$$ which simplifies to 
\begin{equation}\label{eq-P_Cs}
 P_{C_s}(x)=\left[\frac12(\sqrt{1+4x}+1)\right]^s+\left[\frac12(\sqrt{1+4x}-1)\right]^s.
\end{equation}

Using this the equation $P_{C_s}(\bar \gamma)=(1+\delta)^s$ can be written as: 
\begin{align*}
\left(1+\sqrt{1+4 \bar \gamma}\right)^s \left[1+\left(\frac{\sqrt{1+4 \bar \gamma}-1}{\sqrt{1+4 \bar \gamma}+1}\right)^s\right]=(2(1+\delta))^s.
\end{align*}
This shows $\eta:=\lim_{s\rightarrow \infty} \bar \gamma(C_s, \delta)$ must satisfy the equation $1+\sqrt{1+4 \eta}=2(1+\delta)$, which solves to $\eta=\delta(1+\delta)$. Therefore, taking limit as $s \rightarrow \infty$, in \eqref{eq:lb_I} gives 
\begin{align}\label{eq:ub}
\liminf_{n \rightarrow \infty} \frac{\phi_1(n, p, 1+\delta)}{n^2p^2\log(1/p)}   & \geq \min\left\{\frac{(1+\delta)^2}{2}, \delta(1+\delta) \right\},
\end{align}
using $\lim_{s\rightarrow \infty}\left((1+\delta)^s-1\right)^{\frac{2}{s}}=(1+\delta)^2$. This completes the proof of the lower bound. 

\subsection{Proof of the Upper Bound in \eqref{eq:ev1_variational}} 
\label{sec:ub_e1}

The proof of the upper bound proceeds by verifying that the minimum entropy configurations are attained by planting a {\it clique} or an {\it anti-clique} of the required size in the Erd\H os-R\'enyi graph. \\ 

\noindent{{\it The Clique Construction}:} For $b\geq 1$, denote by  $G_{\mathrm{cl}}(b)$ the {\it clique graph}, with adjacency matrix $A(G_{\mathrm{cl}}(b))=((a_{ij}))$, which has zeros on the diagonal and 
\begin{align}\label{cliquecons}
a_{ij}=
\left\{
\begin{array}{ccc}
1  &  \text{ if } 1 \leq i \ne j \leq \lceil b\rceil+1 \\
p  &   \text{otherwise.} 
\end{array}
\right.
\end{align}
(Note that this corresponds to planting a clique on the set of vertices $\{1,2,3, \ldots, \lceil b\rceil+1\}$ in the random graph $\cG_{n, p}$.) Next, define  
\begin{equation}
\label{eq:evcandidate}
\bm v=\left(\underset{ \lceil b\rceil +1 \text{ times}}{\underbrace{\frac{1}{ \sqrt{\lceil b\rceil +1}},\ldots, \frac{1}{\sqrt{\lceil b\rceil+1}}}}, \underset{ n-\lceil b\rceil - 1 \text{ times}}{\underbrace{0,\ldots,0}}\right)',
\end{equation}
Now, let $z:=1+\delta$, and choosing $b=znp$ gives, (since $\bm v' \bm v =1$) 
\begin{align}\label{eq:ev_clique}
\lambda_1(G_{\mathrm{cl}}(znp)) \geq \bm v'A(G_{\mathrm{cl}}(znp)) \bm v & =\sum_{1 \leq i \ne j \leq \lceil z np\rceil +1} a_{ij}v_iv_j  \nonumber \\ 
& =\frac{1}{\lceil z np\rceil +1}\sum_{1 \leq i \ne j \leq \lceil z np\rceil +1} a_{ij} \nonumber \\ 
& =\lceil z np\rceil \geq z np =(1+\delta) np. 
\end{align} 
This shows, 
\begin{align}\label{eq:ub_I}
\phi_1(n, p, 1+\delta) \leq I_p(G_{\mathrm{cl}}(znp)) \leq (1+o(1))  \tfrac{1}{2}(1+\delta)^2 n^2 p^2 \log(1/p).
\end{align}\\

\noindent{{\it The Anti-clique Construction}:} Given $z \geq 0$, $G_{\mathrm{acl}}(z)$ be weighted graph which has an anti-clique on the vertices $\{1,2,3, \ldots, \lfloor znp^2 \rfloor\}$ and weight $p$ between every other pair of vertices, that is, $A(G_{\mathrm{acl}}(z))=((a_{ij}))$, which has zeros on the diagonal and 
$$a_{ij}=
\left\{
\begin{array}{ccc}
1  &  \text{ if } 1 \leq i \leq \lfloor znp^2 \rfloor \text{ and } 1\leq j \leq n \\
p  &   \text{otherwise.} 
\end{array}
\right.$$
Now, fix $z$ to be chosen later. Choose $$\bm v = K \left(\underset{ \lfloor znp^2 \rfloor \text{ times}}{\underbrace{1, 1, \ldots, 1}},\underset{ n-\lfloor znp^2 \rfloor \text{ times}}{\underbrace{(1+\delta) p, (1+\delta) p,  \ldots, (1+\delta) p}}\right)',$$
where $K=\frac{1}{\sqrt{  \lfloor znp^2 \rfloor +(n-\lfloor znp^2 \rfloor) (1+\delta)^2 p^2}}$. This ensures $\bm v' \bm v=1$. Then 
\begin{align}\label{eq:lambda1_ac}
\lambda_1(G_{\mathrm{acl}}(z)) & \geq \bm v'A(G_{\mathrm{acl}}(z)) \bm v \\ 
& \geq  2 K^2 (1+\delta) p \lfloor znp^2 \rfloor (n-\lfloor znp^2 \rfloor) +  K^2 (1+\delta)^2 p^3 (n-\lfloor znp^2 \rfloor)(n-\lfloor znp^2 \rfloor-1) \nonumber \\
& = K^2 n^2p^3\left(  2  (1+\delta)  z  (1-  zp^2 ) +  (1+\delta)^2  (1- zp^2)^2 \right) +o(np)\nonumber \\ 
& =  (1+\delta)  np \left[\frac{   2 z  (1-  zp^2 ) +  (1+\delta) (1- zp^2)^2}{z+(1+\delta)^2(1-zp^2)} \right] +o(np). 
\end{align} 
Note that if we set $z=\delta(1+\delta)$ then, since $zp^2 \ll 1$, $\frac{   2 z  (1-  zp^2 ) +  (1+\delta) (1- zp^2)^2}{z+(1+\delta)^2(1-zp^2)}=1+o(1)$, and the RHS of \eqref{eq:lambda1_ac} becomes $np (1+\delta) + o(np)$. Thus by choosing a slightly larger $z=(1+o(1))\delta(1+\delta)$, we get, 
$$\lambda_1(G_{\mathrm{acl}}(z)) \geq (1+\delta)  np.$$ This shows 
\begin{align}\label{eq:ub_II}
\phi_1(n, p, 1+\delta) \leq I_p(G_{\mathrm{acl}}(z)) \leq (1+o(1)) \delta(1+\delta) n^2 p^2 \log(1/p).
\end{align}

Combining \eqref{eq:ub_I} and \eqref{eq:ub_II} gives the desired upper bound $$\phi_1(n, p, 1+\delta)  \leq (1+o(1))\min\left\{\frac{(1+\delta)^2}{2}, \delta(1+\delta) \right\} n^2p^2\log(1/p),$$ completing the proof.

\begin{remark} The mean-field approximation to the upper tail probability for $\lambda_1(\cG_{n, p})$, as in Theorem \ref{thm:e1_ldp}, is expected to hold for $p \gg \frac{\log^\alpha n}{n}$, for some $\alpha >0$.  Therefore, in anticipation of this result, it makes sense to ask what happens to the variational problem $\phi_1(n, p,1+\delta)$, for $n^{-1} \ll p \ll n^{-\frac{1}{2}}$. However, unfortunately, the sparsity of the graph renders the cycle homomorphism density irrelevant, because, in this regime, it is determined by the smaller eigenvalues and is no longer related to the largeness of $\lambda_1(\cG_{n, p})$. This is reflected in the fact that in \eqref{homonotiso}, in the said sparsity regime, $t(C_s,\cG_{n,p})$ is much larger than the number of injective homomorphisms of $C_s$ into $\cG_{n,p}$ (where the sum in \eqref{homonotiso} is restricted to only distinct indices). It is worth pointing out that although not directly related to spectral information, current technology can be used to establish (see \cite{BGLZ}), that for any $\theta >1$ and $n^{-1} \ll p \ll n^{-\frac{1}{2}}$,
$$\tilde \phi(C_s,n, p , \theta)= (1+o(1)) \tfrac12 (\theta-1)^{\frac{2}{s}} n^2p^2\log(1/p),$$ where $\tilde \phi$ is the analogous variational problem to \eqref{eq:disvar} for the injective homomorphism. Nonetheless, the analog of Theorem \ref{thm:ev1} in the regime $n^{-1} \ll p \ll n^{-\frac{1}{2}}$, continues to remain open. Unlike cliques and anti-cliques, we expect that structures such as a single high degree vertex will govern the large deviation behavior of $\lambda_1(\cG_{n, p})$ in this case.  
\end{remark}

\section{Proof of Theorem \ref{thm:ev1_centered}}
\label{sec:pfev1_centered}

Note that it suffices to prove \eqref{eq:ev1_centered_var}. The proof of the upper bound is presented below in Section \ref{sec:ub_centered}, and the proof of the lower bound is given in Section \ref{sec:lb_centered}.

\subsection{Proof of the Upper Bound in \eqref{eq:ev1_centered_var}}
\label{sec:ub_centered}

 To establish an upper bound on \eqref{eq:ev1_centered_var}, it suffices to exhibit a weighted graph $G_n$ with $||A(G_n)-p \bm 1 \bm 1'||_{\mathrm{op}} \geq \delta np$, which has entropy $(1+o(1))  \frac{1}{2}\delta^2 n^2 p^2 \log(1/p)$. For this, consider the clique graph $G_{\mathrm{cl}}(s)$ with adjacency matrix $A(G_{\mathrm{cl}}(s))=((a_{ij}))$ as in \eqref{cliquecons}, with $s=\frac{\delta np + 2p}{1-p}=(1+o(1)) \delta np$ (since $\delta n \geq \delta np \rightarrow \infty$). This ensures, $\lceil s\rceil - p (\ceil{s}+1) \geq \delta np$. Then choosing  $\bm v$ as in \eqref{eq:evcandidate}, gives 
\begin{align*}
||A(G_n)-p \bm 1 \bm 1'||_{\mathrm{op}}  \geq \bm v'(A(G_n) -p \bm 1 \bm 1') \bm v &= \sum_{1 \leq i \ne j \leq \lceil s\rceil +1} a_{ij}(G_n)v_iv_j - p (\ceil{s}+1) \nonumber \\
& =\frac{1}{\lceil s\rceil +1}\sum_{1 \leq i \ne j \leq \lceil s\rceil +1} a_{ij}(G_n) -p (\ceil{s}+1) \nonumber \\
&=\lceil s\rceil - p (\ceil{s}+1) \geq \delta np.
\end{align*}
This shows (recall \eqref{eq:disvarlambda}), $\phi_2(n, p, \delta np) \leq I_p(G_n) \leq (1+o(1))  \frac{1}{2}\delta^2 n^2 p^2 \log(1/p)$.

\subsection{Proof of the Lower Bound in \eqref{eq:ev1_centered_var}} 
\label{sec:lb_centered}

For the lower bound, it is convenient to analyze a continuous version of the variational problem~\eqref{eq:disvar}. To describe the 
continuous analogue of the variational problem, it is convenient to invoke the language of graph limit theory~\cite{BCLSV08,BCLSV12,Lov12,LS06}. Define a \emph{signed graphon} as a  symmetric measurable function $W : [0, 1]^2\rightarrow [-1, 1]$ (where symmetric means $W(x,y) = W(y,x)$). Every graphon $W$ defines an operator $T_W: L_2[0, 1]\rightarrow L_2[0, 1]$, by 
\begin{eqnarray}
(T_Wf)(x)=\int_0^1W(x, y)f(y) \mathrm dy.
\label{eq:T}
\end{eqnarray}
$T_W$ is a Hilbert-Schmidt operator, which is compact and has a discrete spectrum, that is, a countable multiset of non-zero real eigenvalues. Denote by $||W||_{\mathrm{op}}$ the operator norm of $T_W$.

In the continuous variational problem, the edge-weighted graph $G_n$ of~\eqref{eq:disvar}, is replaced by a non-negative signed graphon $W$, that is, $W \in [0, 1]$, for all $x, y \in [0, 1]$ ($G_n$ should be viewed as a discrete approximation of a non-negative signed graphon.) Denote the space of all non-negative signed graphons (hereafter, simply referred to as graphons) by $\cW$. Finally, defining
$
\E[f(W)] := \int_{[0,1]^2} f(W(x,y))\, \mathrm d x \mathrm d y
$, we get the continuous analogue of \eqref{eq:disvarlambda}.

\begin{defn}[Graphon variational problem]
For $\theta>0$ and $0<p<1$, let
\begin{equation}
\phi_2(p, \theta):=\inf\Bigl\{\tfrac12 \E[I_p(W)] : \text{ $W \in \cW$ with } ||W-p J||_{\mathrm{op}} \geq \theta \Bigr\},\
\label{eq:continuous}
\end{equation}
where $J:=1$, is the constant graphon 1. 
\end{defn}

The following lemma shows how we can obtain a lower bound on the discrete variational problem using the graphon variational problem: 

\begin{lem} \label{lem:disccont} 
Under the assumptions of Theorem \ref{thm:ev1_centered}, $\phi_2(p, (1+o(1))\delta p)\le \frac{1}{n^2}\phi_2(n, p, \delta np)$.
\end{lem}

\begin{proof}
Given weighted graph $G_n \in \sG_n$ with adjacency matrix $(a_{ij})_{1\le i,j \le n}$, form two graphons $W^{G_n} $ and $\hat W^{G_n}$ as follows: Divide $[0,1]$ into $n$ equal-length intervals $I_1,I_2, \ldots, I_n$, that is, $I_i=(\frac{i-1}{n}, \frac{i}{n}]$, and set (1) $\hat W^{G_n}(x,y)=a_{ij}$ if $x\in I_i, y \in I_j$ and $i\neq j$, and $\hat W^{G_n}(x, y)=0$ if $x,y \in I_i$ for some $i$, and (2) $W^{G_n}(x,y)=a_{ij}$ if $x\in I_i, y \in I_j$ and $i\neq j$, and $W^{G_n}(x, y)=p$ if $x,y \in I_i$ for some $i$. 

Note that the operator norm of $\hat W^{G_n}-p J$ restricted to act on the set of step functions that are constant on the intervals $\{I_i\}_{i=1}^n$ is exactly $\frac{1}{n} ||A(\cG_{n,p})-p \bm 1 \bm 1'||_{\mathrm{op}}$. This implies, 
\begin{align}
 ||A(\cG_{n,p})-p \bm 1 \bm 1'||_{\mathrm{op}} & \le  n ||\hat W^{G_n}-p J||_{\mathrm{op}} \nonumber \\
& \leq n ||W^{G_n}-p J||_{\mathrm{op}}  + n ||\hat W^{G_n}- W^{G_n}||_{\mathrm{op}}  \nonumber \\
&\leq n ||W^{G_n}-p J||_{\mathrm{op}} + p, \nonumber 
\end{align}
where final inequality follows by observing that $\hat W^{G_n}- W^{G_n}$ is the graphon which is $p$ on $I_i\times I_i$, for $1 \leq i \leq n$ and zero otherwise, and checking that $||\hat W^{G_n}- W^{G_n}||_{\mathrm{op}}=p/n$. 

This shows, if $||A(\cG_{n,p})-p \bm 1 \bm 1'||_{\mathrm{op}}  \geq \delta np$, then $||W^{G_n}-p J||_{\mathrm{op}} \geq \delta p-p/n=(1+o(1)) \delta p$ (since $\delta n \geq \delta np \rightarrow \infty$). The result follows by noting that  $\frac{1}{2}I_{p}(W^{G_n})= \frac{1}{n^2} I_p(G_n)$ (diagonal entries contribute $0$ to $I_{p}(W^{G_n})$).
\end{proof}

We need another lemma which states that the infimum in \eqref{eq:continuous} can be restricted to graphons which are pointwise at least $p.$
\begin{lem}\label{lm:restrict23} Let $\phi_2(p, \theta)$ be as defined in \eqref{eq:continuous}. Then 
\begin{equation}
\phi_2(p, \theta) \geq \inf\Bigl\{\tfrac12 \E[I_p(W)] : \text{ $W \in \cW$ with } W\ge p \text{ and } ||W-p J||_{\mathrm{op}} \geq \theta \Bigr\}.\
\end{equation}
\end{lem}

\begin{proof} Let $W=p+U$, where $-p \leq U \leq 1-p$.  Denote by $\hat U:=|U|$ the graphon obtained by taking the absolute value of $U$. Then clearly from the definition \eqref{eq:T}, it follows that  $||\hat U||_{\mathrm{op}}\ge ||U||_{\mathrm{op}}$.  The proof is now complete by Lemma \ref{lm:Ip} which guarantees  $I_p(p+x) \geq I_p(p+|x|)$, for all $p\le 1/2$ and $p+x\ge0.$ 
\end{proof}

Given the above discussion, the proof of the lower bound in \eqref{eq:ev1_centered_var} is a straightforward consequence of the following lemma: 

\begin{lem}\label{lm:phi2} For $p\ll 1$ and $\delta=O(1)$, $\phi_2(p, \delta p) \geq (1+o(1))  \tfrac{1}{2}\delta^2 p^2 \log(1/p)$.
\end{lem}

\begin{proof} Suppose  $W \in \cW$ such that $||W-p J||_{\mathrm{op}} \geq \delta p$ and $W\ge p$. Then, by Corollary \ref{est3}, 
$$\delta^2 p^2  \le ||W-p J||_{\mathrm{op}}^2 \le   \int_0^1 \int_0^1 (W(x, y)-p)^2 \mathrm dx \mathrm d y \le (1+o(1))  \frac{1}{I_p(1)} \int_0^1 \int_0^1 I_p(W(x, y)) \mathrm dx \mathrm d y,$$
which in conjunction with Lemma \ref{lm:restrict23}, completes the proof.
\end{proof}

\section{Proof of Corollary \ref{cor:ev2}}
\label{sec:pfev2}

The upper bound in Corollary \ref{cor:ev2} is an immediate consequence of Theorem \ref{thm:ev1_centered} and the following simple consequence of Weyl's interlacing inequality (see \cite{taotopics}): $\lambda_2(A(\cG_{n,p}))\le ||A(\cG_{n, p})-p \bm 1 \bm 1'||_{\mathrm{op}}.$

{The lower bound follows by planting a clique of appropriate size in the random graph.} Given $b \geq 0$, let $\cC_b$ be the event that the vertices $\{1,2,3, \ldots, \lceil b\rceil+1\}$ forms a clique in $\cG_{n, p}$. Note that 
\begin{align}
\P(\lambda_2(\cG_{n, p}) \geq \delta np) & \geq \P((\lambda_2(\cG_{n, p}) \geq \delta np) \cap \cC_{\delta np}) \nonumber \\ 
&=\P(\lambda_2(\cG_{n, p}) \geq \delta np | \cC_{\delta np}) \cdot \P(\cC_{\delta np}) \nonumber \\
&=\frac{1}{2}e^{-(1+o(1))\frac{1}{2} \delta^2 n^2p^2 \log(1/p)}, \nonumber 
\end{align}
using $\P(\cC_{\delta np})=p^{\lceil \delta np \rceil +1 \choose 2}$ and the lemma below. This completes the proof of the lower bound.

\begin{lem} Under the assumption of Corollary \ref{cor:ev2}, $\P(\lambda_2(\cG_{n, p}) \geq \delta np | \cC_{\delta np})  \geq \frac{1}{2}$.
\end{lem}

\begin{proof}  Throughout the proof, $\bm v$ will denote the vector in \eqref{eq:evcandidate} with $b=\delta np$. Note that $\bm v' \bm 1 = (\ceil{b}+1)^{\frac{1}{2}}$, which means
\begin{align}\label{eq:vbar}
\overline{\bm v}: = \frac{(\ceil{b}+1)^{\frac{1}{2}}}{n} \cdot \bm 1
\end{align}  
is the projection of $\bm v$ on $\bm 1$, and $\hat {\bm v}={\bm v}-\overline{\bm v}$, which is a vector orthogonal to $\bm 1$. 
Next, recall the following well-known variational representation of the second eigenvalue of the adjacency matrix $A(\cG_{n, p})=((a_{ij}))$: 
\begin{align}\label{eq:eigenvalue2}
\lambda_2(\cG_{n, p})=\sup_{V:~\mathrm{dim}(V)=2} \left\{\inf _{\bm w\in V} \frac{{\bm w}'A(\cG_{n, p}){\bm w}}{{\bm w}'{\bm w}} \right\}.
\end{align} 
To prove the lemma, we will show that given $\cC_{\delta n p}$, the following hold with probability at least $\frac{49}{50}$: 
\begin{align}\label{conditions1}
{\bm 1}'A(\cG_{n, p}){\bm 1}  \ge (1+o(1))  n^2p, \quad  \hat {\bm v}'A(\cG_{n, p})\hat {\bm v}  &\ge \delta np (1+o(1)), 
\end{align} 
and
\begin{align}\label{condition2}
|{{\bm 1}'A(\cG_{n, p})\hat {\bm v}}|\le O(\delta^{\frac{1}{2}}(np)^{\frac{3}{2}}).
\end{align} 
To see why this suffices, choose $V$ to be the 2-dimensional subspace generated by $\bm 1$ and $\hat {\bm v}$ in \eqref{eq:eigenvalue2} and note that for $\bm w=\alpha {\bm 1}+ \beta \hat{\bm  v}$, 
\begin{align*}
\frac{{\bm w}'A(\cG_{n, p}){\bm w}}{{\bm w}'{\bm w}}&=\frac{\alpha^2 {\bm 1}'A(\cG_{n, p}){\bm 1}+\beta^2\hat {\bm v}'A(\cG_{n, p})\hat {\bm v}+2 \alpha \beta {\bm 1}'A(\cG_{n, p}) \hat{\bm v}}{\alpha^2 {\bm 1'}{\bm 1}+\beta^2 \hat{\bm v}'\hat{\bm v}}\\
&\ge \frac{\alpha^2 (1+o(1))n^2p+\beta^2\delta np (1+o(1))-2\alpha \beta O(\delta^{\frac{1}{2}} (np)^{\frac{3}{2}})}{\alpha^2 n+\beta^2 } \tag*{(using $\hat {\bm v}'\hat {\bm v}\le {\bm v}'{\bm v} =1$)} \\
&\ge \frac{\alpha^2 (1+o(1))n^2p+\beta^2\delta np (1+o(1))-O(\alpha^2 n^{2}p^{\frac{3}{2}})- O(\beta^2\delta np^{\frac{3}{2}})}{\alpha^2 n+\beta^2 } \\
&\ge \frac{\alpha^2 (1+o(1))n^2p+\beta^2\delta np (1+o(1))}{\alpha^2 n+\beta^2 } \tag*{(using $p \rightarrow 0$)}\nonumber \\
& \ge (1+o(1))\delta np, 
\end{align*}
where the last step uses $\delta <1$, and the third step uses the AM-GM inequality $$\alpha^2 n^{2}p^{\frac{3}{2}}+\beta^2\delta np^{\frac{3}{2}}\ge 2\alpha \beta \delta^{\frac{1}{2}} (np)^{\frac{3}{2}}.$$ Therefore, replacing $\delta$ by a larger $\delta(1+o(1))$ above, along with \eqref{eq:eigenvalue2},  ensures, $\lambda_2(\cG_{n, p}) \geq \delta np$, given $\cC_{\delta(1+o(1)) n p}$, with probability at least $\frac{49}{50}$.

We now proceed to prove \eqref{conditions1} and \eqref{condition2}. Start by observing that given $\cC_{\delta np}$,  
\begin{align*}
(\delta n p)^2+\sum_{i=1}^n \sum_{\substack{1 \leq j \leq n \\ j \ne i}}a_{ij} \ge \bm 1'A(\cG_{n, p}) \bm 1  \geq \sum_{i=1}^n \sum_{\substack{\lceil b\rceil +2 \leq j \leq n,\\ j \ne i}}a_{ij},  
\end{align*} 
where $a_{ij}$, for $1 \leq i < j \leq n$, are i.i.d. Bernoulli$(p).$
Now, using simple Binomial concentration it follows that with probability at least $\frac{99}{100}$  
\begin{equation}\label{binomialbound}
\bm 1'A(\cG_{n, p}) \bm 1 = (1+o(1)) n^2p,
\end{equation} 
which implies the first inequality in \eqref{conditions1}.  For the second inequality in \eqref{conditions1}, observe that
\begin{align}\label{eq:ev2_I}
\hat {\bm v}'A(\cG_{n, p})\hat {\bm v}& = {\bm v}'A(\cG_{n, p}){\bm v}-2\overline{\bm v}'A(\cG_{n, p}){\bm v} + \overline{\bm v}'A(\cG_{n, p})\overline{\bm v} \nonumber \\
& \geq {\bm v}'A(\cG_{n, p}){\bm v}-2\overline{\bm v}'A(\cG_{n, p}){\bm v} \nonumber  \\
&\ge \delta n p- \tfrac{2(\ceil{b}+1)^{\frac{1}{2}} }{n}{\bm 1}'A(\cG_{n, p}){\bm v}, 
\end{align}
using ${\bm v}'A(\cG_{n, p}){\bm v} \geq \delta n p$, as $b = \delta np$
(as in \eqref{eq:ev_clique}), and \eqref{eq:vbar}. Then, recalling the definitions of $\bm v$ and $\overline{\bm v}$ from \eqref{eq:evcandidate} and \eqref{eq:vbar}, respectively, it follows that 
\begin{equation}\label{relation34}
(\ceil{b}+1)^{\frac{1}{2}} \left\{{\bm 1}'A(\cG_{n, p}){\bm v} \right\}=n \left\{\overline{\bm v}'A(\cG_{n, p}){\bm v} \right\}= \bm 1'A(\cG_{n, p}){\bm 1}_{\ceil{b}+1},
\end{equation}
where 
$${\bm 1}_{\ceil{b}+1} := \left(\underset{ \lceil b\rceil +1 \text{ times}}{\underbrace{1, 1, \ldots, 1}}, \underset{ n-\lceil b\rceil - 1 \text{ times}}{\underbrace{0, 0, \ldots,0}}\right)'.$$  
Now, on the event $\cC_{\delta np}$, 
\begin{align}\label{eq:ev2_II}
\bm 1'A(\cG_{n, p}){\bm 1}_{\ceil{b}+1} \leq  &  \lceil b\rceil  (\lceil b\rceil +1)  +  \sum_{1 \leq i  \leq \lceil b\rceil +1}  \sum_{1 \leq j \leq n} a_{ij} \nonumber \\ 
& = O(\delta^2 n^2 p^2) +   \sum_{1 \leq i  \leq \lceil b\rceil +1}  \sum_{1 \leq j \leq n} a_{ij}.
\end{align}
Finally, by Markov's inequality, the second term above is less than $$ 100 \E \sum_{1 \leq i  \leq \lceil b\rceil +1}  \sum_{1 \leq j \leq n} a_{ij} =O(n b p)= O(\delta n^2 p^2),$$ with probability at least $\frac{99}{100}$. 
Thus with probability at least $\frac{99}{100},$ 
\begin{align}\label{finalbound}
\bm 1'A(\cG_{n, p}){\bm 1}_{\ceil{b}+1}&= O(\delta n^2p^2), 
\end{align}
Assuming the above event, and using \eqref{relation34}, we conclude that in \eqref{eq:ev2_I}, 
$${\hat {\bm v}'A(\cG_{n, p})\hat {\bm v}}\ge \delta np-O(\delta n p^2)=(1+o(1))\delta np,$$
which proves the second inequality in \eqref{conditions1}.

Finally, recalling \eqref{eq:vbar} and $b=\delta n p$, gives   
\begin{align*}
{ {\bm 1}'A(\cG_{n, p})\hat {\bm v}} & =  {\bm 1}'A(\cG_{n, p}){\bm v}-{\bm 1}'A(\cG_{n, p})\overline{\bm v}  = {\bm 1}'A(\cG_{n, p}){\bm v}- \tfrac{(\ceil{b}+1)^{\frac{1}{2}} }{n}{\bm 1}'A(\cG_{n, p}){\bm 1} \le O(\delta^{\frac{1}{2}} (np)^{\frac{3}{2}}),
\end{align*}
where the bound on the first term above follows from \eqref{relation34} and \eqref{finalbound}, and the same bound on the second term follows from 
\eqref{binomialbound}. This proves \eqref{condition2} and completes the proof of the lemma. 
\end{proof} 

\section{Discussions and Future directions}
\label{sec:openproblems}

The results in this paper were concerned with the analysis of variational problems for the upper tails of edge eigenvalues of $\cG_{n, p}$. The corresponding problem for the lower tail is also interesting, which for the case of the largest eigenvalue is well understood. It follows from results in Lubetzky and Zhao \cite[Proposition 3.9]{LZ-dense} for the dense case and, more generally, from Cook and Dembo \cite[Theorem 1.21]{CD18}, that the lower tail problem for $\lambda_1(\cG_{n, p})$ exhibits replica symmetry, that is, 
the corresponding variational problem is minimized by the constant function. More precisely, it is known that, for $\frac{\log n}{n} \ll p \leq \frac{1}{2}$ and $0 < q < p $ (such that $s:=q/p \in (0, 1)$ is fixed), 
$$\P(\lambda_1(\cG_{n, p})  \leq q(n-1))=e^{-(1+o(1)) {n \choose 2} I_p(q)}.$$

Finally, this work leaves many questions open, especially for $\lambda_2(\cG_{n,p})$ and other smaller edge eigenvalues. We list below a few of them.

\begin{itemize}
\item[(1)] Extending Corollary \ref{cor:ev2} to the case $\delta>1$. The key issue one faces is that $\lambda_2(\cG_{n,p}) >np$ automatically guarantees that $\lambda_1(\cG_{n,p})>np$ as well, and in particular the Perron-Frobenius eigenvector will now be not close to the constant vector. Hence, one cannot automatically transfer the knowledge about the operator norm of $A(\cG_{n,p})-p \bm1\bm1'$ to the second eigenvalue of $A(\cG_{n,p})$. 

\item[(2)] Establishing a joint large deviation for cycle homomorphism densities of different sizes, and using these, or otherwise, obtain a joint LDP for $\lambda_1(\cG_{n,p})$ and $\lambda_2(\cG_{n,p})$, or of various spectral moments. 

\item[(3)] Compute large deviation probabilities for the $k$-th largest eigenvalue $\lambda_k(\cG_{n,p})$.  Does the upper tail large deviation behavior of $\lambda_k(\cG_{n,p})$ agree with the probability of planting $k$ small cliques of appropriate sizes (up to negligible factors in the exponential scale) in some regime of the parameter space?
\end{itemize}

\section{Appendix: Estimates for the Entropy Function}

In this section, we collect a few relevant estimates for $I_p(x)$ from~\cite{LZ-sparse}, which we used in our proofs. The $\sim$ notation below is with respect to limits as $p\to 0$.

\begin{lem}[{\cite[Lemma~3.3]{LZ-sparse}}]\label{est1}
If $0\le  x \ll p$, then $I_p(p + x) \sim \frac12 x^2/p$, whereas when $p \ll x \le 1- p$ we have
$I_p(p + x) \sim x \log(x/p)$.
\end{lem}

\begin{lem}[{\cite[Lemma~3.4]{LZ-sparse}}]\label{est2} There is some constant $p_0 > 0$ such that for every $0 < p \le p_0$,
\[ I_p(p + x)\ge {(x/b)}^{2}I_p(p + b)\qquad\mbox{ for any $0\le x \le b\le 1-p- \log(1-p)$}\,.\]
\end{lem}

\begin{cor}[{\cite[Corollary 3.5]{LZ-sparse}}]\label{est3} There is some constant $p_0 > 0$ such that for every $0 < p \le p_0$,
\[I_p(p + x) \ge x^2I_p(1 - 1/ \log(1/p)) \sim x^2I_p(1)\qquad\mbox{ for any  $0 \le x \le 1 - p$}\,.\] 
\end{cor} 

We conclude with the following lemma, which is required in the proof of Lemma \ref{lm:restrict23}. 

\begin{lem}\label{lm:Ip}
For {$p \leq \frac{1}{2}$} and $0 \leq x \leq p$, $I_p(p-x) \geq I_p(p+x)$. 
\end{lem}

\begin{proof} Define $F_p(x):=I_p(p-x) - I_p(p+x)$. It is easy to check that 
$$F'_p(x)=\log\left(\frac{p-x-1}{p-1}\right)-\log\left(\frac{p-x}{p}\right)-\log\left(\frac{p+x}{p}\right)+\log\left(\frac{p+x-1}{p-1}\right).$$ Now, using $F_p'(0)=0$ and $F_p''(x)=\frac{2(1-2p)x}{(1-p+x)(p-x)(1-p-x)(p+x)}\geq 0$, {for $p \leq \frac{1}{2}$}, implies 
$F_p'(\cdot)$ is non-decreasing, that is, $F_p'(x)\geq F_p(0)=0$. This, in turn implies that $F_p(\cdot)$ is non-decreasing, that is, $F_p(x) \geq F_p'(0)=0$. 
\end{proof}

\small \noindent{\it Acknowledgment}: The authors thank Debapratim Banerjee, Nick Cook, and Amir Dembo for useful discussions. The authors also thank the anonymous referees for their  detailed and thoughtful comments, which greatly improved the quality and presentation of the paper.

\bibliographystyle{abbrv}
\bibliography{ldp_ref}

\begin{thebibliography}{10}

\bibitem{AKV}
N.~Alon, M.~Krivelevich, and V.~H. Vu.
\newblock On the concentration of eigenvalues of random symmetric matrices.
\newblock {\em Israel Journal of Mathematics}, 131(1):259--267, 2002.

\bibitem{aging}
G.~B. Arous, A.~Dembo, and A.~Guionnet.
\newblock Aging of spherical spin glasses.
\newblock {\em Probability Theory and Related Fields}, 120(1):1--67, 2001.

\bibitem{freeentropy}
G.~B. Arous and A.~Guionnet.
\newblock Large deviations for {W}igner's law and {V}oiculescu's
  non-commutative entropy.
\newblock {\em Probability Theory and Related Fields}, 108(4):517--542, 1997.

\bibitem{augeri2}
F.~Augeri.
\newblock Large deviations principle for the largest eigenvalue of wigner
  matrices without gaussian tails.
\newblock {\em Electron. J. Probab.}, 21:49 pp., 2016.

\bibitem{augeri1}
F.~Augeri.
\newblock Nonlinear large deviation bounds with applications to traces of
  wigner matrices and cycles counts in {E}rd{\H{o}}s-{R}\'{e}nyi graphs.
\newblock {\em arXiv:1810.01558}, 2018.

\bibitem{austin}
T.~Austin.
\newblock The structure of low-complexity gibbs measures on product spaces.
\newblock {\em arXiv preprint arXiv:1810.07278}, 2018.

\bibitem{ab_rb}
A.~Basak and R.~Basu.
\newblock Upper tail large deviations of the cycle counts in
  {E}rd{\H{o}}s-{R}\'{e}nyi graphs in the full localized regime.
\newblock {\em arXiv:1912.11410}, 2019.

\bibitem{BM17}
A.~Basak and S.~Mukherjee.
\newblock Universality of the mean-field for the potts model.
\newblock {\em Probability Theory and Related Fields}, 168(3-4):557--600, 2017.

\bibitem{BGLZ}
B.~B. Bhattacharya, S.~Ganguly, E.~Lubetzky, and Y.~Zhao.
\newblock Upper tails and independence polynomials in random graphs.
\newblock {\em Advances in Mathematics}, 319(313--347), 2017.

\bibitem{BGSZ}
B.~B. Bhattacharya, S.~Ganguly, X.~Shao, and Y.~Zhao.
\newblock Upper tails for arithmetic progressions in a random set.
\newblock {\em International Mathematics Research Notices}, to appear, 2018.

\bibitem{bordenave1}
C.~Bordenave and P.~Caputo.
\newblock A large deviation principle for wigner matrices without gaussian
  tails.
\newblock {\em The Annals of Probability}, 42(6):2454--2496, 2014.

\bibitem{BCLSV08}
C.~Borgs, J.~T. Chayes, L.~Lov{\'a}sz, V.~T. S{\'o}s, and K.~Vesztergombi.
\newblock Convergent sequences of dense graphs. {I}. {S}ubgraph frequencies,
  metric properties and testing.
\newblock {\em Adv. Math.}, 219(6):1801--1851, 2008.

\bibitem{BCLSV12}
C.~Borgs, J.~T. Chayes, L.~Lov{\'a}sz, V.~T. S{\'o}s, and K.~Vesztergombi.
\newblock Convergent sequences of dense graphs {II}. {M}ultiway cuts and
  statistical physics.
\newblock {\em Ann. of Math. (2)}, 176(1):151--219, 2012.

\bibitem{CD16}
S.~Chatterjee and A.~Dembo.
\newblock Nonlinear large deviations.
\newblock {\em Adv. Math.}, 299:396--450, 2016.

\bibitem{CV11}
S.~Chatterjee and S.~R.~S. Varadhan.
\newblock The large deviation principle for the {E}rd{\H o}s-{R}{\'e}nyi random
  graph.
\newblock {\em European J. Combin.}, 32(7):1000--1017, 2011.

\bibitem{CV12}
S.~Chatterjee and S.~R.~S. Varadhan.
\newblock Large deviations for random matrices.
\newblock {\em Comm. Stoch. Analysis}, 6(1):1--13, 2012.

\bibitem{CD18}
N.~Cook and A.~Dembo.
\newblock Large deviations of subgraph counts for sparse
  {E}rd{\H{o}}s-{R}\'enyi graphs.
\newblock {\em arXiv:1809.11148}, 2018.

\bibitem{satya1}
D.~S. Dean and S.~N. Majumdar.
\newblock Large deviations of extreme eigenvalues of random matrices.
\newblock {\em Physical Review Letters}, 97(16):160201, 2006.

\bibitem{Eldan}
R.~Eldan.
\newblock Gaussian-width gradient complexity, reverse log-{S}obolev
  inequalities and nonlinear large deviations.
\newblock {\em Geom. Funct. Anal.}, to appear, 2018.

\bibitem{yau}
L.~Erd{\H{o}}s, A.~Knowles, H.-T. Yau, and J.~Yin.
\newblock Spectral statistics of {E}rd{\H{o}}s-{R}\'{e}nyi graphs ii:
  Eigenvalue spacing and the extreme eigenvalues.
\newblock {\em Communications in Mathematical Physics}, 314(3):587--640, 2012.

\bibitem{guionnet1}
A.~Guionnet and J.~Husson.
\newblock Large deviations for the largest eigenvalue of rademacher matrices.
\newblock {\em arXiv:1810.01188}, 2018.

\bibitem{upper_tail_localization}
M.~Harel, F.~Mousset, and W.~Samotij.
\newblock Upper tails via high moments and entropic stability.
\newblock {\em arXiv preprint arXiv:1904.08212}, 2019.

\bibitem{Lov12}
L.~Lov{{\'a}}sz.
\newblock {\em Large networks and graph limits}, volume~60 of {\em American
  Mathematical Society Colloquium Publications}.
\newblock American Mathematical Society, Providence, RI, 2012.

\bibitem{LS06}
L.~Lov{\'a}sz and B.~Szegedy.
\newblock Limits of dense graph sequences.
\newblock {\em J. Combin. Theory Ser. B}, 96(6):933--957, 2006.

\bibitem{LZ-dense}
E.~Lubetzky and Y.~Zhao.
\newblock On replica symmetry of large deviations in random graphs.
\newblock {\em Random Structures Algorithms}, 47(1):109--146, 2015.

\bibitem{LZ-sparse}
E.~Lubetzky and Y.~Zhao.
\newblock On the variational problem for upper tails in sparse random graphs.
\newblock {\em Random Structures Algorithms}, 50(3):420--436, 2017.

\bibitem{ev_concentration}
G.~Lugosi, S.~Mendelson, and N.~Zhivotovskiy.
\newblock Concentration of the spectral norm of {E}rd{\H{o}}s-{R}\'{e}nyi
  random graphs.
\newblock {\em arXiv:1801.02157}, 2018.

\bibitem{taotopics}
T.~Tao.
\newblock {\em Topics in random matrix theory}, volume 132.
\newblock American Mathematical Soc., 2012.

\bibitem{satya2}
P.~Vivo, S.~N. Majumdar, and O.~Bohigas.
\newblock Large deviations of the maximum eigenvalue in {W}ishart random
  matrices.
\newblock {\em Journal of Physics A: Mathematical and Theoretical},
  40(16):4317, 2007.

\bibitem{warnke_missing_log}
L.~Warnke.
\newblock On the missing log in upper tail estimates.
\newblock {\em Journal of Combinatorial Theory, Series B}, 140:98--146, 2020.

\bibitem{yan}
J.~Yan.
\newblock Nonlinear large deviations: Beyond the hypercube.
\newblock {\em arXiv preprint arXiv:1703.08887}, 2017.

\bibitem{Zhao-lower}
Y.~Zhao.
\newblock On the lower tail variational problem for random graphs.
\newblock {\em Combin. Probab. Comput.}, 26(2):301--320, 2017.

\end{thebibliography}

\end{document}